\documentclass[11pt]{amsart}

\usepackage{fullpage}
\usepackage{amssymb, amsmath, wasysym, mathrsfs, mathtools, color, stmaryrd, enumerate}
\usepackage[all, cmtip]{xy}
\usepackage{url}
\usepackage{tikz-cd}

\definecolor{hot}{RGB}{65,105,225}
\usepackage[pagebackref=true,colorlinks=true, linkcolor=hot ,  citecolor=hot, urlcolor=hot]{hyperref}

\newcommand{\commentOut}[1]{}

\theoremstyle{plain}
\newtheorem{theorem}{Theorem}[section]
\newtheorem{prop}[theorem]{Proposition}

\newtheorem{thrm}[theorem]{Theorem}

\theoremstyle{definition}

\newtheorem{defn}[theorem]{Definition}
\newtheorem{rmk}[theorem]{Remark}

\newtheorem*{ex*}{Example}

\newtheorem{subs}[theorem]{ }

\newcommand{\cO}{{\mathcal O}}

\newcommand{\cM}{{\mathcal M}}

\newcommand{\cV}{{\mathcal V}}

\newcommand{\bZ}{{\mathbb{Z}}}

\newcommand{\bC}{{\mathbb{C}}}
\newcommand{\bA}{{\mathbb{A}}}
\newcommand{\bN}{{\mathbb{N}}}

\def\bz{{\mathbf{0}}}

\newcommand{\ubul}{{\,\begin{picture}(-1,1)(-1,-3)\circle*{2}\end{picture}\ }}

\DeclareMathOperator{\id}{id}

\DeclareMathOperator{\pic}{Pic}

\DeclareMathOperator{\rank}{rank}

\DeclareMathOperator{\enmo}{End}
\DeclareMathOperator{\spec}{Spec}

\DeclareMathOperator{\Def}{Def}

\DeclareMathOperator{\length}{length}                    
\DeclareMathOperator{\ord}{ord}                    
\DeclareMathOperator{\Hom}{Hom}

\DeclareMathOperator{\HH}{H}

\DeclareMathOperator{\MC}{MC}

\DeclareMathOperator{\mld}{mld}

\def\ra{\rightarrow}

\def\al{\alpha}

\newcommand\be{\begin{equation}}
\newcommand\ee{\end{equation}}

\def\Linf{L_\infty}

\def\Art{{\mathcal{A}rt}}
\def\Set{{\mathcal{S}et}}
\def\eps{\epsilon}

\def\mm{{\mathfrak{m}}}

\def\wh{\widehat}
\def\pa{\partial}
\def\length{{\rm{length}}}

\title[Local structure of theta divisors and related loci of generic curves]{Local structure of theta divisors and related loci of generic curves}

\author{Nero Budur}\address{Department of Mathematics, KU Leuven, Celestijnenlaan 200B, 3001 Leuven, Belgium} \email{nero.budur@kuleuven.be}

\begin{document}

\begin{abstract} 
For a generic compact Riemann surface the theta function is at every point on the Jacobian equal to its first Taylor term, up to a holomorphic change of local coordinates and multiplication by a local holomorphic unit. More generally, any Brill-Noether locus of twisted stable vector bundles on a smooth projective curve is at every point $L$ locally \'etale isomorphic with its tangent cone if the Petri map at $L$ is injective. This assumption has  various consequences for Brill-Noether loci: positive answers to the monodromy conjecture for generalized theta divisors and to questions of Schnell-Yang on log resolutions and Whitney stratifications, and formulas for local $b$-functions, log canonical thresholds, topological zeta functions, and minimal discrepancies.
\end{abstract}

\maketitle

\setcounter{tocdepth}{1}

\numberwithin{equation}{section}

\tableofcontents

\section{Introduction} 

The original motivation for this work came from Schnell-Yang \cite[Problem 9.2]{SY} who asked whether the Brill-Noether stratification of the theta divisor of a generic curve is a Whitney stratification, and whether by blowing up the strata one at a time in increasing dimension one achieves a log resolution of singularities. Another motivation was the curiosity to understand if the deformation theory with cohomology constraints from \cite{BW, BR} can lead to something new in a classical subject such as Brill-Noether loci of curves. In {\it loc. cit.} this theory was applied to seemingly more complicated objects, but it is supposed to be the right theory to deal  with local aspects of all cohomology jump  loci. Using this theory  positive answers to both questions from above are given here. In fact the main result, Theorem \ref{thrmLCThrk} computes many singularity invariants of theta divisors of generic curves, and provides generalizations to twisted and higher-rank Brill-Noether loci. 

These results follow from the  characterization of the local germs of these loci as linear determinantal varieties given in Theorem \ref{thmGenToConeEt}. The answers to the two questions from above on the theta divisor of generic curves follow from the more elementary version for classical Brill-Noether loci of line bundles, Theorem \ref{thrmWrd}. After the first version of this article it was pointed out to me that Theorem \ref{thrmWrd} is known or at least expected by some experts although there seems to be no reference available. A more elementary proof based on the global determinantal description of classical Brill-Noether loci of line bundles suffices to prove Theorem \ref{thrmWrd}. This does not work for twisted Brill-Noether loci as there is in general no such determinantal description even Zariski locally. Thus for Theorem \ref{thmGenToConeEt} one needs  the full power of the deformation theory with cohomology constraints. This proof boils down to the more elementary one for Theorem \ref{thrmWrd}. Next, the results are stated in detail.

Let $C$ be a smooth projective curve over an algebraically closed field $K$ of characteristic zero. 
Consider  the Brill-Noether locus 
$$
W_d^r:=\{L\in\pic^d(C)\mid h^0(L)>r\}
$$
 of isomorphism classes of line bundles $L$ of degree $d$ with at least $r+1$ linearly independent global sections, endowed with the natural scheme structure, where $r\in\bN$.  Let 
 $$
\pi_{L}:\HH^0(C,L)\otimes \HH^0(C,\omega_C\otimes L^{-1})\to \HH^0(C,\omega_C)
 $$
 be the Petri map of $L$, where  $\omega_C$ is the canonical sheaf of $C$. Regarding the tangent cones of $W^r_d$  one has:
 
\begin{thrm} {\rm{(}\cite{Ke}, \cite[VI, Thm. 2.1]{A+}\rm{)}} Let $L\in\pic^d(C)$ with $0\neq h^0(L)h^1(L)$. 
If $\pi_L$ is injective, the tangent cone $TC_LW_d^r$    is the closed 
  subscheme defined by the ideal generated by the minors of size $h^0(L)-r$ of the  $h^1(L)\times h^0(L)$ matrix of linear forms on $\HH^1(C,\cO_C)$ given by  $\pi_{L}$. (For out-of-bounds minors see \ref{subNot} for  definition.)
\end{thrm}


The local embedded structure of $W_d^r\subset \pic^d(C)$ is completely determined if $\pi_L$ is injective.
Denote by $\bz$  the origin of $\HH^1(C,\cO_C)$ identified with the tangent space of $\pic^d(C)$ at $L$.

\begin{thrm}\label{thrmWrd} If $\pi_L$ is injective there is a local $K$-isomorphism for the \'etale topology between $(\pic^d(C),L)$ and  $(\HH^1(C,\cO),\bz)$ inducing    a local $K$-isomorphism for the \'etale topology between $(W_d^r,L)$
 and $(TC_LW_d^r,\bz)$.
\end{thrm}

This implies the corresponding statement for local analytic germs in the case $K=\bC$. This seems known or at least expected by some experts. If $C$ is a generic curve with fixed genus then $\pi_L$ is always injective by  \cite{Gie}. Hence the above theorems apply to generic curves. A particular case is:

\begin{thrm}\label{thmThetaJac} At every point of the Jacobian  $J(C)$ of a generic compact Riemann surface $C$, there exists a holomorphic change of coordinates such that Riemann's theta function $\theta$ is equal to its first Taylor term  times a holomorphic unit.
\end{thrm}

 To state a generalization to twisted stable vector bundles we first fix notation. Let $g$ be the genus of the   smooth projective curve $C$  over the field $K$.  Let $n>0, d\ge 0, k> 0$ in $\bZ$. Fix a vector bundle $F$ on $C$. Let $\cM_{n,d}$ be  the moduli space of stable vector bundles on $C$ of rank $n$ and degree $d$. Let
$$
\cV_{n,d,k}(F):=\{E\in \cM_{n,d}\mid h^{0}(C,E\otimes F)\ge k\}
$$
be endowed with the natural structure of closed subscheme of $\cM_{n,d}$. 
 Set $\cV_{n,d,k}=\cV_{n,d,k}(\cO_C)$. When $\cM_{n,d}$ is fixed from the context, we set $$\cV_k(F)=\cV_{n,d,k}(F)\quad\text{ and }\quad\cV_k=\cV_k(\cO_C).$$  Then $\cV_k(F)$ form a filtration of closed subschemes
 $
 \cM_{n,d}=\cV_0(F)\supset\cV_1(F)\supset\cV_2(F)\supset\ldots.
 $
  For $E\in\cM_{n,d}$ the {\it Petri map} is the natural multiplication map
$$
\pi_{E,F}: \HH^0(C,E\otimes F)\otimes \HH^0(C,E^\vee\otimes F^\vee\otimes \omega_C) \to \HH^0(C,E\otimes E^\vee\otimes\omega_C)
$$
Set $\pi_{E}=\pi_{E,\cO}$. Set 
$l=h^0(C,E\otimes F)$, $l'=h^1(C,E\otimes F)$. Then 
\be\label{eqChi}
l-l'=\chi(E\otimes F) =n\deg(F)-\rank (F)(n(g-1)-d).
\ee
We assume $l\ge 1$. If $E$ is a line bundle, that is $n=1$, we denote it by $L$ to stress this fact. Note that when $n=1$, $\cM_{1,d}=\pic^d(C)$ and $\cV_k=W^{k-1}_d$.

\begin{theorem}\label{thmGenToConeEt} There is a canonical isomorphism of $K$-vector spaces between the tangent space $T_E\cM_{n,d}$ and  $\HH^1(C,E\otimes E^\vee)$. Assume $\pi_{E,F}$ is  injective. Then:
\begin{itemize}
\item There is a local $K$-isomorphism for the \'etale topology between $(\cM_{n,d},E)$ and  $(\HH^1(C,E\otimes E^\vee),\bz)$ inducing  for every $1\le k\le l$ local $K$-isomorphisms for the \'etale topology between $(\cV_{k}(F),E)$
 and $(TC_E\cV_k(F),\bz)$.
\item The tangent cone $TC_E\cV_{k}(F)$ is the closed 
  subscheme defined by the ideal generated by the minors of size $l-k+1$ of the  $l'\times l$ matrix of linear forms on $\HH^1(C,E\otimes E^\vee)$ given by $\pi_{E,F}$. 
  \end{itemize}
 \end{theorem}
 
Only the first part seems new, for the description of the tangent cone in the second part see  \cite{Po}.
 The injectivity condition holds for example in the following cases:
  
 \begin{theorem}\label{thrmInj}
Assume that $C$ is generic among curves with same genus and
\begin{itemize}
\item $\rm{(}$\cite{Gie}$\rm{)}$   $F=\cO_C$, or
\item $\rm{(}$\cite{TiBtw}$\rm{)}$   $F$ is generic among vector bundles with same rank and degree.
\end{itemize}
Then the Petri map
$
\pi_{L,F}$
is injective for every $L\in \pic^d(C)$.
\end{theorem}

 To prove Theorem \ref{thmGenToConeEt} we prove first its formal analog, Theorem \ref{thmGenToCone}. Theorem \ref{thmGenToCone} is only a slightly improved version of
 \cite[Thm. 0.1]{Po} where   $A_\infty$-categories where used, the difference  being that here compatible isomorphisms of formal neighborhoods for the whole filtration of closed subschemes $\{\cV_k(F)\}_{k\ge 0}$ are claimed, although this might already follow from {\it loc. cit.} In any case, we prove Theorem \ref{thmGenToCone} using the deformation theory with cohomology constraints of \cite{BW} in its version in terms of $\Linf$ pairs \cite{BR}. The proof is  almost a tautology.  A second proof of Theorem \ref{thmGenToCone} via $\Linf$ pairs bearing similarity to \cite[Thm. 3.1]{Po} is delegated to \cite[\S 13]{BD} since it requires more $\Linf$ background. 
 
To pass from  formal  to   \'etale local coordinates, a version for pairs of Artin's algebraization theorem is used,  Proposition \ref{propTower}.

By  Theorem   \ref{thmGenToConeEt} the \'etale local models for Brill-Noether loci on curves at stable bundles with injective Petri maps are generic determinantal schemes. This has many consequences.
To state them it is convenient to assume  that
\be\label{eqAssu}
 l=h^0(C,E\otimes F)\le l'=h^1(C,E\otimes F).
\ee
By (\ref{eqChi}) this assumption is independent of $E$. This is for simplicity since one can always reduce to this case. If $F=\cO_C$, (\ref{eqAssu}) becomes
$
n(g-1)-d\ge 0.
$
If $E=L$ is a line bundle and $F=\cO_C$, this is equivalent to $d<g$. The   singularity theory terms  below are reviewed in \cite[\S 9]{BD}.

\begin{thrm}\label{thrmLCThrk}\label{thrmLCT} In the setup of Theorem \ref{thmGenToConeEt}, let $K=\bC$,
let   $E\in \cV_k(F)\subset\cM_{n,d}$ with $1\le k\le l$, satisfying (\ref{eqAssu}), and such that $\pi_{E,F}$ is injective.
Then the following hold in a   Zariski open neighborhood of $E$ in $\cM_{n,d}$:

\begin{enumerate}[(i)]

\item\label{qo} $\cV_k(F)$  is variety with at most rational singularities, it has dimension 
$$
\rho_{n,d,k}(F) := n^2(g-1)+1-k\bigl( k - n\deg(F)+\rank (F)(n(g-1)-d) \bigr),
$$
and  the singular locus of $\cV_k(F)$ is $\cV_{k+1}(F)$.





\item\label{qiii} If $k=1$ the local Bernstein-Sato polynomial at $E$ of the ideal defining  $\cV_1(F)$ in $\cM_{n,d}$ is
$$
\prod_{i=l'-l+1}^{l'} (s+i).
$$

\item\label{qix} The log canonical threshold of $(\cM_{n,d},\cV_k(F))$ at $E$ is $$
\min\left\{\frac{(l-i)(l'-i)}{l -k+1-i}\mid i=0,\ldots ,l-k\right\}.
$$

\item\label{qvi} Consider $f:Y\to \cM_{n,d}$ the composition of blowups of (strict transforms) of $\cV_l(F)$, $\cV_{l-1}(F)$, $\cV_{l-2}(F)$, $\ldots$, $\cV_1(F)$, in this order. Then:

\begin{itemize}
\item At each stage this is the blowup of a smooth center. 
\item The composition $f$ is a log resolution of $(\cM_{n,d},\cV_{k}(F))$. 

\item The pullback of the ideal sheaf defining $\cV_{k}(F)$ is $\cO_Y(-\sum_{i=0}^{l-k}(l-k+1-i)E_i)$, where $E_i$ is the (strict transform of the) divisor introduced by blowing up the (strict transform of) $\cV_{l-i}(F)$.
\end{itemize}

\item\label{qvii} The stratification of $\cV_k(F)$ given by $\cV_{t}(F)\setminus \cV_{t+1}(F)$ with $k\le t$ is a Whitney stratification, and the local Euler obstruction at $E$ of $\cV_k(F)$  is $\binom{l}{l-k}$.


\item\label{qiv} If $d=n(g-1-\deg(F)/\rank (F))$, equivalently $l=l'$,   the topological zeta function at $E$ of the pair $(\cM_{n,d},\cV_k(F))$ is $$\prod_{\al\in \Omega}\frac{1}{1-\al^{-1}s}$$
where
$$
\Omega\subset\left\{ -\frac{l^2}{l-k+1}, -\frac{(l-1)^2}{l-k}, -\frac{(l-2)^2}{l-k-1},\ldots, - k^2 \right\}.
$$

\item\label{qv} If $k=1$ and $d=n(g-1-\deg(F)/\rank (F))$, the  monodromy conjecture  holds  locally at $E$ for the generalized theta divisor $\cV_1(F)\subset\cM_{n,d}$, that is, the local Bernstein-Sato polynomial times the local topological zeta function is a polynomial in $s$.

\item\label{xi} If $l=l'$, the minimal discrepancies of $\cV_k(F)$ along $\cV_{k+1}(F)$ and, respectively, along a point $E'\in \cV_{k'}(F)\setminus \cV_{k'+1}(F)$ with $k\le k'\le l$ are:
$$
\mld(\cV_{k+1}(F);\cV_k(F)) = k+1,\quad\quad \quad\mld(E';\cV_k(F)) = l^2-kk'.
$$
\end{enumerate}
 \end{thrm}

Part   (\ref{qo}) is a more general version of  results  due to  \cite{Ke, AC, A+, TiBtw, CT}.

Part (\ref{qiii}) implies that the local minimal exponent of $\cV_1(F)$ at $E$ is 2, which answers positively a more general form of a question raised in an early version of \cite{SY2}.

Part (\ref{qix}) recovers a result of   \cite{Zhu} for $n=1$ and $F=\cO_C$, and extends it to a more general form. The proof from \cite{Zhu} is long since it does not make use of something like Theorem \ref{thrmWrd}. Our proof explains what was remarked as a curiosity in \cite{Zhu} that the log canonical threshold   equals that of a  determinantal variety.

Parts (\ref{qvi}), (\ref{qvii}) answer positively  \cite[Problem 9.2]{SY}, and in fact they answer a more general version of this problem.  Part (\ref{qvi}) also recovers a result of \cite[Thm. 3.3]{Mul} saying that for $n=1$, $F=\cO_C$, and $d=g-2$, this blowup process is an embedded resolution (without checking the simple normal crossings condition) of $(\pic^{g-2}(C),W^0_{g-2})$.

See Part I of the survey \cite{BD} for other consequences regarding multiplicities, multiplier ideals, Hodge ideals, jet schemes, and local cohomology of Brill-Noether loci.

\subs{\bf Organization.}  
In Section \ref{secDFT} we review the deformation theory with cohomology constraints of \cite{BW, BR} in terms of $\Linf$ pairs. Here we include only a  black box. A technical review of   
$\Linf$ algebras and $\Linf$ modules can be found in the survey article \cite{BD}. Section \ref{secDFT} ends with the proof of Theorem \ref{thmGenToCone}, the formal analog of Theorem \ref{thmGenToConeEt}.

In Section \ref{secApplBN} we prove Theorem \ref{thmGenToConeEt}. In Section \ref{secDet} we collect some known facts about the singularities of spaces of generic matrices. In Section \ref{Proof} we prove Theorem \ref{thrmLCThrk}.

\subs{\bf Acknowledgement.} I thank  M. Coppens, C. Chiu,  R. Docampo, A.K. Doan, M. Rubi\'o, C. Schnell, R. Yang, N. Zhang for discussions.
This work was partially supported by the Methusalem grant METH/21/03 and the grants G097819N, G0B3123N from FWO.

\subs{\bf Notation.}\label{subNot}  A {\it  variety} $X$ over a field $K$ of characteristic zero is a geometrically irreducible, reduced, separated scheme of finite type over $K$. 
By convention,  the ideal generated by the $k$-minors of an $l'\times l$ matrix of linear forms, with $k\in\bZ$ and $l,l'>0$, is the zero ideal if $k>\min\{l,l'\}$, and is the ideal $\langle 1\rangle$ if $k\le 0$.

\section{Deformations with cohomology constraints}\label{secDFT}

 Deformation problems with cohomology constraints over a field $K$ of characteristic zero are controlled by dgl pairs, or better, $L_\infty$ pairs, according to a generalization of Deligne's principle by  \cite{BW, BR, BD}.
  A pair means an algebra together with a module.  The main feature of the theory is that only one  pair is necessary to control the local structure at a fixed point $L$ of the pairs $(\cM,\cV_k^i)$ for all $k,i \in\bZ$, where $\cM$ is a fixed moduli scheme of objects with a cohomology theory and $\cV_k^i=\{L\in\cM\mid h^i(L)\ge k\}$ are the cohomology jump  subschemes. Equivalent  pairs describe the same $(\cM,\{\cV_k^i\}_{i,k})$ locally at $L$. 
The moduli space $\cM$ does not have to exist: in general one has a deformation functor for $L$ and deformation subfunctors for all $i, k$ that play the role of the formal completions $\widehat{\cM}_L$ and $\widehat{(\cV^i_k)}_L$.  
   

\subs{\bf Dgl pairs.}
Classical deformation theory studies the local structure of $\cM$ at  $L$ by  attaching a differential graded Lie algebra (dgla) $C$ such that $$
\widehat{\cM}_L\simeq \Def ({C}) :\Art \ra \Set
$$ 
as functors from the category $\Art$ of local Artinian finite type $K$-algebras to the category of sets. 
 The deformation functor  $\Def(C)$ is defined by associating to every $A$ in $\Art$ with maximal ideal $\mathfrak{m}_A$ the set of Maurer-Cartan elements of $C\otimes \mm_A$ modulo the gauge action
$$
\Def(C; A) :=\{\omega\in C^1\otimes_K \mathfrak{m}_A\mid d_C\omega+\frac{1}{2}[\omega,\omega]_C=0\}/ (C^0\otimes_K\mathfrak{m}_A),
$$
where $d_C$ is the differential of $C$ extended by identity on $A$, and $[.\,,.]_C$ is the Lie bracket of $C$ extended by the usual multiplication on $A$. Two quasi-isomorphic dgla's have isomorphic deformation functors \cite{GM}.

To the object $L$ one also attaches a (left) dgl  module $M$ over $C$, that is,  a dgl pair  $(C,M)$ in the terminology from \cite{BW}.   When $C, M$ are graded by $\bN$, and as cochain complexes they are bounded-above and have finite-dimensional cohomology, one has well-defined cohomology jump deformation subfunctors $\Def^i_k(C,M)$ of $\Def(C)$ for   $i,k\in\bZ$ such that
$$
\widehat{(\cV^i_k)}_L\simeq \Def^i_k(C,M),
$$
and any two equivalent dgl pairs give the same cohomology jump deformation subfunctors, by \cite[\S 3]{BW}. In this case we say that $(C,M)$  {\it controls} the deformations of $L$ with cohomology constraints. The deformation subfunctors send $A\in\Art$ to
\begin{align*}
\Def^i_k(C,M;A)  := \{\omega\in C^1\otimes_K & \mathfrak{m}_A\mid  d_C\omega+\frac{1}{2}[\omega,\omega]_C=0\text{ and }\\
& J^i_k(M\otimes_KA, d_M+\omega)=0 \}/ (C^0\otimes_K\mathfrak{m}_A),
\end{align*}
where $d_M$ is the differential of $M$ extended by identity on $A$ to $M\otimes_KA$, and the cohomology jump ideals $J^i_k\subset A$ of the  complex $(M\otimes_KA, d_M+\omega)$ of $A$   modules are defined as follows.

\begin{defn}\label{defCJI} 
 Let $R$ be a noetherian commutative ring and  $N$ a complex of $R$-modules, bounded above, with finitely generated cohomology. There always exists a bounded above complex $F$ of finitely generated free $R$-modules and a quasi-isomorphism of complexes $F \xrightarrow{\sim} N$. The {\it cohomology jump ideals} of $N$ are the ideals in $R$ defined as 
$$
J^i_k(N) = I_{{\rank}(F^i)-k+1}(d^{i-1} \oplus d^i),
$$
where $d^i: F^i \rightarrow F^{i+1}$ are the differentials of $F$, and $I_r$ is the ideal generated by the  $r\times r$ minors. The cohomology jump ideals do not depend on the choice of the free resolution, by \cite[\S 2]{BW}.
\end{defn}

\begin{prop}\label{rmkPair} 
Let $E, F$ be two vector bundles over a smooth projective variety $X$ over an algebraically closed field $K$. Assume that $E$ is stable with respect to a fixed polarization. Then the deformations of $E$ with cohomology constraints $h^i(X,E\otimes F)\ge k$ are controlled by the dgl pair $( R\Gamma(X, \enmo  (E)), R\Gamma(X,E\otimes F))$.
\end{prop}
\begin{rmk}
These complexes of derived global sections are not uniquely defined, so let us specify what we mean. First note that independently of the model chosen for these complexes,  the Lie bracket is induced by the graded commutator of the composition in the endomorphisms vector bundle $\enmo  (E)$, and the action of $\enmo(E)$ on $E\otimes F$ lifts to the model. In each case below, it can be checked that the model is a dgl pair. 

When $K=\bC$ and one takes as model the pair of the Dolbeault resolutions of $\enmo(E)$ and $E\otimes F$, the claim is proven in \cite[Theorem 6.4]{BW}. For arbitrary $K$, it is well-known that  the deformations of $E$ are also controlled by a \v{C}ech complex $\Check{C}(\mathcal{U},\enmo (E))$, and that there is an equivalence between this dgla and the Dolbeault dgla in the complex case, see \cite[Rmk. 5.1]{FMM}. By a similar argument one has in the complex case the equivalence between \v{C}ech and Dolbeault pairs. To show when $K\neq \bC$ that the \v{C}ech pair $( \Check{C}(\mathcal{U},\enmo (E)),  \Check{C}(\mathcal{U}, E \otimes F) )$ controls the deformations of $E$ with constraints as above, one adapts the proof of \cite[Theorem 6.4]{BW} where the  ``local Kuranishi family of vector bundles"  is now provided by ``quasi-universal family" of \cite[Prop. 4.6.2]{HL}. 

One also obtains that  the scheme structure of $\widehat{\cV_k(F)}_E$ described by $\Def^0_k$ of this dgl pair agrees with the classical one in the case $X$ is a curve, since the local quasi-universal family is also used to define the latter, cf. \cite{CT}.

\end{rmk}

\subs{\bf $\Linf$ pairs.} A  more efficient theory has been developed in \cite{BR} by passing from dgl pairs to $\Linf$ pairs. The category of dgla's is a subcategory of  the category of  $\Linf$ algebras, and the category of dgl pairs is a subcategory of the category of $\Linf$ pairs, that is, pairs consisting of an $\Linf$ algebra together with an $\Linf$ module, with morphisms appropriately defined.
We refer to \cite[\S 12]{BD} for  details and definitions. Here we only recall that an {\it $\Linf$ algebra} is a graded vector space $ C$ together with a collection of graded anti-symmetric multilinear maps
$$
l_n: C^{\otimes n}\ra  C
$$
of degree $2-n$ for every $n\ge 1$, satisfying a generalized Jacobi identity. The map $l_1$ is a differential, so that $(C,l_1)$ is a complex. Dgla's are   $\Linf$ algebras with $l_n=0$ for $n\ge 3$, in which case $l_1$ is the differential and $l_2$ is the Lie bracket. An {\it $\Linf$ module} over $ C$ is a graded vector space $ M$ together with a collection of graded linear maps
$$
m_n: C^{\otimes n-1}\otimes  M\ra  M
$$
of degree $2-n$ for every $n\ge 1$, satisfying a certain compatibility with the maps $l_n$. The map $m_1$ is a differential, so that $(M,m_1)$ is a complex. 

There are two other equivalent definitions of the $\Linf$ pair structure: one in terms of graded symmetric multilinear maps, and one in terms of codifferentials on a comodule over a coalgebra \cite[12.17]{BD}. The latter is the most conceptual definition. The notion of {\it weak equivalence} between $\Linf$ pairs is recalled in  \cite[\S 12.45]{BD}.

The homotopy transfer theorem  \cite[Thm. 12.46]{BD} says that for a dgl pair $(C,M)$ the cohomology graded vector spaces $HC$ and $HM$ can be endowed with an $\Linf$ algebra structure $l_*$ and, respectively, an $\Linf$ module structure $m_*$, such that: $l_1=0$, $m_1=0$, $l_2$ and $m_2$ are induced from  the Lie bracket and the dgl module structure, and the dgl pair $(C,M)$ is weakly equivalent as an $\Linf$ pair with $(HC,HM)$. The transfer depends on some choices. Regarding the dgla $(C,d_C,\left[\_\;,\_\right]_C)$ one chooses a {homotopy retract}
 \be\label{eqHtr}
\xymatrix{
C \ar@(lu,ld)_h \ar@<.5ex>[r]^p & \ar@<.5ex>[l]^\iota HC
}
\ee that is, $p:(C,d_C)\ra(HC,0)$ and $\iota:(HC,0)\ra (C,d_C)$ are morphisms of cochain complexes,  $\iota$ is a quasi-isomorphism, and  $h:C\ra C[-1]$ is a graded linear map such that $\id_C-\iota p=d_Ch+hd_C$. Then the $\Linf$ multiplications maps $l_n$ on $HC$ are determined from the homotopy retract using rooted binary trees, see \cite[Thm. 12.31]{BD}. There is also a choice to bemade involving $M$. We will not need the explicit formulas for $l_n$ and $m_n$.

There is a well-defined deformation functor $\Def (HC)$ for the $\Linf$ algebra $HC$ such that to  every $A$ in $\Art$ one attaches
 \be\label{eqLMC}
\Def(HC; A):=\Bigl\{\omega\in H^1C\otimes_K \mathfrak{m}_A\mid \sum_{n\ge 2}\frac{1}{n!}l_n^A(\omega^{\otimes n})=0\Bigr\}/_\sim
\ee
where $l_n^A=l_n\otimes\id_A$ and $\sim$ is the homotopy equivalence relation \cite[Def. 12.40]{BD}. Homotopy transfer \cite[Thm. 12.31]{BD} implies that the existence of the second isomorphism of functors
\be\label{eqMOD}
\widehat{\cM}_L\simeq \Def (C)\simeq \Def (HC)
\ee
This is due to Fukaya, Kontsevich, Soibelman,
Manetti, etc.,  \cite[Thm. 12.42]{BD}.

It was shown in \cite{BR} that, under the assumption that the cochain complex $M$ is bounded above, there are well-defined subfunctors $\Def^i_k(HC,HM)$ of $\Def(HC)$ such that
\begin{equation}\label{eqLJI}
\begin{split}
\Def^i_k(HC,HM;A):=&\left\{\omega\in  H^1C\otimes_K \mathfrak{m}_A\mid \sum_{n\ge 2}\frac{1}{n!}l_n^A(\omega^{\otimes n})=0\text{ and }\right.\\
&\left.J^i_k\biggl(HM, \sum_{n\ge 1}\frac{1}{n!}m_{n+1}^A(\omega^{\otimes n}\otimes \_)\biggr)=0 \right\}/_\sim
\end{split}
\end{equation}
with the cohomology jump ideals $J^i_k\subset A $ defined as above and $m_n^A:=m_n\otimes\id_A$,  \cite[Def. 12.50]{BD}. A weak equivalence of  $\Linf$ pairs induces an isomorphism of deformation functors restricting to isomorphisms of the cohomology jump deformation subfunctors, see \cite[Thm. 12.53]{BD}. This implies the existence of the second   isomorphism of functors
\be\label{eqTISOM}
\widehat{(\cV^i_k)}_L\simeq \Def^i_k(C,M)\simeq \Def^i_k(HC,HM),
\ee
see \cite[Thm. 12.53]{BD}. A basic result is the determination of the tangent spaces to these functors. Recall that the  tangent space to a deformation functor $F$ is $TF:=F(K[\eps]/(\eps^2))$.

\begin{theorem}{{\rm{(}}\cite[Thm 1.7]{BR}{\rm{)}}}
\label{propTDef}
Let $(C,M)$ be a dgl pair or, more generally, an $\Linf$ pair, over a field of characteristic zero. Assume that $C, M$ are $\bN$-graded and that $M$ is bounded above as a cochain complex. Let $h_i=\dim H^iM$. The  tangent spaces to the functors 
$$\Def^i_0(C,M)=\Def(C)\supset\ldots\supset\Def^i_k(C,M)\supset\ldots \supset  \Def^i_{h_i+1}(C,M)=\emptyset$$ 
are: the full  tangent space $T\Def(C)=H^1C$ if $k<h_i$; empty if $k>h_i$; and if $k=h_i$, equal to the kernel of the linear map
$$
H^1C\ra \bigoplus_{j=i-1,i}\Hom(H^{j}M,H^{j+1}M)
$$
induced from the $\Linf$ module multiplication maps
$
H^1C\otimes H^jM\ra H^{j+1}M.
$
\end{theorem}

Applying Proposition \ref{rmkPair} and  homotopy transfer \cite[Thm. 12.46]{BD} one has:

\begin{prop}\label{propOurLePot}
Let $E, F$ be two vector bundles over a smooth projective variety $X$ over an algebraically closed field $K$. Assume that $E$ is stable with respect to a fixed polarization. Then:

\begin{enumerate}[(1)]

\item The deformations of $E$ with cohomology constraints $h^i(E\otimes F)\ge k$ are controlled by the $\Linf$ pair $(\HH^\ubul(X,\enmo(E)), \HH^\ubul(X,E\otimes F))$.

\item If $\cM$ denotes the moduli space of stable vector bundles on $X$ of same Hilbert polynomial as $E$,  $\cV^i_k=\{E'\in\cM\mid h^i(E\otimes F)\ge k\}$ denote the cohomology jump loci endowed with the natural closed subscheme structure, and $h^i=h^i(E\otimes F)$, then the Zariski tangent spaces at $E$ to
$$\cV^i_0=\cM\supset\ldots\supset\cV^i_k\supset\ldots \supset  \cV^i_{h_i+1} (=\emptyset \text{ around } E)$$ 
are: the full Zariski tangent space $T_E\cM=\HH^1(X,\enmo(E))$ if $k<h_i$; empty if $k>h_i$; and if $k=h_i$, equal to the kernel of the linear map
$$
\HH^1(X,\enmo(E))\ra \bigoplus_{j=i-1,i}\Hom(\HH^{j}(X,E\otimes F),\HH^{j+1}(X,E\otimes F))
$$
induced from the natural multiplication maps
$
\HH^1(X,\enmo(E))\otimes \HH^{j}(X,E\otimes F)\ra \HH^{j+1}(X,E\otimes F).
$
\end{enumerate}

\end{prop}

The second part is classical for Brill-Noether loci of line bundles when $X$ is a curve, cf. \cite[IV, Prop. 4.2]{A+}.

Until now we used $(C,M)$ to denote a dgl or $\Linf$ pair as in \cite{BW,BR}. From now we find it more suggestive to denote by $(M,V)$ such a pair controlling the local structure of $(\cM,\{\cV^i_k\}_{i,k})$.
Next result  will be used to prove Theorem \ref{thmGenToCone}.

\begin{thrm}\label{thm1GDik}  Let $(M,V)$ be an $L_\infty$ algebra together with a module, both of finite dimension over a field $K$ of characteristic zero, such that:
\begin{itemize} 
\item $M^i=0$ and $V^i=0$ for $i\neq 0,1$,
\item the differentials on $M$ and $V$ are zero, 
\item  the linear map $\pi:V^0\otimes (V^1)^\vee\to (M^1)^\vee$ induced from the multiplication map $m_2:M^1\otimes V^0\to V^1$ is injective.
\end{itemize} 

Assume that the $\Linf$ algebra $M$ is  obtained as a homotopy-transferred structure from a dgla $C$ with $\iota:M=HC\subset C$ as in (\ref{eqHtr}),  
and  $[\iota(M^0),C]=0$, where $[\,\_,\_]$ is the Lie bracket of $C$.   Let $\bz\in M^1$ denote the origin. For every $k\in \bN$ let $\cV_k\subset M^1$ be
 the closed subscheme   defined by minors of size $\dim V^0 -k+1$ of the matrix of linear forms on $M^1$ determined by $\pi$. 
 Then there is a canonical isomorphism of vector spaces $T\Def(M)= M^1$ and  an isomorphism of functors $\Def(M)\simeq \widehat{(M^1)}_\bz$ compatible with each other, inducing
 isomorphisms of functors  $\Def^0_k(M,V)\simeq \widehat{(\cV_k)}_\bz$ for every $k$.  
 \end{thrm} 
 \begin{proof}
 Denote by $l=\{l_n\}_{n\ge 1}$ the $\Linf$ algebra structure on $M$, and by $m=\{m_n\}_{n\ge 1}$ the $\Linf$ module structure on $V$. We have $l_1=0$ and $m_1=0$ by assumption.

Let $\omega\in  M^1$. Since $l_n$ has degree $2-n$, $l_n(\omega^{\otimes n})$ is in $M^2=0$.  Hence $M^1\otimes \mm_A=\MC_M(A)$ for all $A\in\Art$, where the Maurer-Cartan set is defined by
$$
\MC_M(A)=\Bigl\{\omega\in M^1\otimes\mm_A\mid \sum_{n\ge 1}\frac{1}{n!}l_n^A(\omega^{\otimes n})=0\Bigr\}.
$$
 This gives $T\Def(M)= M^1$, cf. Theorem \ref{propTDef}. Our assumption on $M^0$  implies by a standard argument that no two elements in $M^1\otimes \mm_A$ are homotopy equivalent    \cite[Lemma 12.43]{BD}. Thus
$
\Def(M)\simeq (\widehat{M^1})_{\bz}. 
$

Since there is no homotopy equivalence to mod out by, we also have
\begin{equation}\label{eqLJIp1}
\Def^0_k(M,V;A)=\{\omega\in  M^1\otimes \mm_A\mid   J^0_k(V\otimes A, d_\omega)=0 \}
\end{equation}
where
$$d_\omega:V^0\otimes A\to V^1\otimes A,\quad
d_\omega(\_):=\sum_{n\ge 1}\frac{1}{n!}m_{n+1}^A(\omega^{\otimes n}\otimes\_),
$$
since $V$ is concentrated in degrees 0,1 and $m_1=0$, cf. (\ref{eqLJI}). It will be slightly more convenient to work with the graded symmetric version of the $\Linf$ pair structure instead of the graded anti-symmetric version. If we keep denoting by $\{m_n\}_n$ the graded symmetric version of the $\Linf$ module structure on $V$, then the formula for $d_{\omega}$ stays the same and we can replace $\omega^{\otimes n}$ by the symmetric self-product which we denote $\omega^{\vee n}$,   \cite[12.37, 12.48]{BD}.

We construct now a universal matrix $d_{univ}$ with entries in the completion $\widehat S$ at the maximal ideal at $\bz\in M^1$ of the symmetric algebra $S$ of $(M^1)^\vee$, such that $d_{univ}$ gives all $d_\omega$ for all  $A$ and $\omega$ as above. Let $s=\dim M^1$.
Fix  a basis $e_1,\ldots, e_s$  of the vector space $M^1$.  Let $x_1,\ldots , x_s$ be the dual basis, so that $S=K[x_1,\ldots ,x_s]$ and $\widehat{S}=K\llbracket x_1,\ldots ,x_s \rrbracket$. Let
$$
\omega_{univ}=\sum_{i=1}^se_i\otimes x_i \quad\in M^1\otimes S.
$$
Define the morphism of free $\widehat{S}$-modules
\be\label{eqDUNI}
d_{univ}:V^0\otimes \widehat{S} \to V^1\otimes \widehat{S},\quad \sigma\otimes 1\mapsto  \sum_{n\geq 1}\frac{1}{n!} (m_{n+1}\otimes\id_{\widehat{S}})((\omega_{univ})^{\vee n}\otimes(\sigma\otimes 1)).
\ee
Fixing bases for $V^0, V^1$, we write $d_{univ}$ as a matrix with entries in $\widehat{S}$.
By construction we have for all $k$ canonical isomorphisms of subfunctors
$$
\Def^0_k(M,V)= {\rm{Spf}}(\widehat{S}/J^0_k(d_{univ}))
$$
compatible with the inclusion of subfunctors for $k\le k'$.

The  matrix $B$ formed by the linear parts of the entries of $d_{univ}$ is by construction the matrix of linear forms on $M^1$ determined by $\pi$ and the above vector space bases. 
By the injectivity assumption on $\pi$, the entries of $B$ are linearly independent. Hence we can find an isomorphism of $\widehat{S}$ such that 
$d_{univ}$ becomes $B$. This implies the claim since $\widehat{(\cV_k)}_\bz$ is defined by the ideal $J^0_k(B)\subset\widehat{S}$. 
 \end{proof}

\begin{rmk}\label{rmk1GDik}
Note that closed subscheme $\cV_k$ of $M^1$  is isomorphic to its tangent cone $TC_\bz\cV_k$ at $\bz$ since $J^0_k(B)$ is a homogeneous ideal.
\end{rmk}

\subs{\bf Formal neighborhoods of Brill-Noether loci.} We prove now the formal neighborhood version of Theorem \ref{thmGenToConeEt}. This is only a slightly improved version of \cite[Thm. 0.1]{Po}. The improvement might be already contained in the proof from {\it loc. cit.}, as mentioned in the introduction. 
 
\begin{theorem}\label{thmGenToCone} Let $E, F$ be as in Theorem \ref{thmGenToConeEt}. There is a canonical isomorphism of $K$-vector spaces between the tangent space $T_E\cM_{n,d}$ and  $\HH^1(C,E\otimes E^\vee)$. If $\pi_{E,F}$ is  injective, there is an isomorphism between the formal neighborhood of $E$ in $\cM_{n,d}$ and the formal neighborhood of the origin in $\HH^1(C,E\otimes E^\vee)$ inducing  for every $1\le k\le l$ isomorphisms between:

\begin{itemize}
\item the formal neighborhood of 
$
\cV_{k}(F)$
 at $E$ in $\cM_{n,d}$,
 \item the formal neighborhood at the vertex of the tangent cone $TC_E\cV_k(F)$  in the tangent space $T_E\cM_{n,d}$.
  \end{itemize}
Moreover,  $TC_E\cV_{k}(F)$ is the closed 
  subscheme defined by the ideal generated by the minors of size $l-k+1$ of the  $l'\times l$ matrix of linear forms on $\HH^1(C,E\otimes E^\vee)$ given by $\pi_{E,F}$. 
 \end{theorem}
 
  \begin{proof}
   By Proposition \ref{rmkPair}, the formal neighborhoods of $E$ in $\cV^i_k(F)$ are controlled by the dgl pair $( R\Gamma(C, \enmo  (E)), R\Gamma(C,E\otimes F))$. By Proposition \ref{propOurLePot}, these formal neighborhoods are controlled by the $\Linf$ pair  $(\HH^\ubul(C,\enmo(E)), \HH^\ubul(C,E\otimes F))$, with an $\Linf$ structure  obtained by homotopy transfer from the dgl pair.  Thus there is isomorphism of functors
$$\widehat{(\cM_{n,d})}_E\simeq \Def(\HH^\ubul(C,\enmo(E)))$$ inducing for all $i,k\in\bN$ an isomorphism of subfunctors $$\widehat{(\cV^i_k)}_E\simeq \Def^i_k(\HH^\ubul(C,\enmo(E)), \HH^\ubul(C,E\otimes F)).$$
We want to apply Theorem \ref{thm1GDik} with $(M,V)=(\HH^\ubul(C,\enmo(E)), \HH^\ubul(C,E\otimes F))$, so we need to check the assumptions. Clearly the dimensional requirements are satisfied since $C$ is a curve. Since $E$ is stable, $\HH^0(C,\enmo(E))=K\cdot \id_E$ is one-dimensional, generated by the identity vector bundle morphism. Since $\id_E$ viewed as an element of the dgla $R\Gamma(C, \enmo  (E))$ commutes with the whole dgla, the condition on $M^0$ is satisfied. Since $(M,V)$ is a transferred structure, the differentials are zero. Note that $\pi:V^0\otimes (V^1)^\vee\to (M^1)^\vee$ is exactly the Petri map $\pi_{E,F}$. Hence the last condition that $\pi$ is injective is part of our hypothesis. Thus Theorem \ref{thm1GDik} and Remark \ref{rmk1GDik} apply. 
\end{proof}

\section{From formal to \'etale neighborhoods}\label{secApplBN}

The goal of this section is to prove  Theorem \ref{thmGenToConeEt}. Let $K$ be a field  of characteristic zero.

\begin{prop}\label{propArSteroids}
Let $X$ be a smooth  $K$-variety, $Y$ a closed subscheme, and $x\in Y$ a point. Let $T=T_xX$, $C=TC_xY$, and $0\in C$ be the vertex. Suppose there exists a $K$-isomorphism of formal neighborhoods $\widehat{X}_x\simeq \wh{T}_0$ inducing an isomorphism $\wh{Y}_x\simeq \wh C_0$. Then:
\begin{itemize}
\item There exists a local isomorphism for the \'etale topology $(X,x)\simeq (T,0)$ inducing a local isomorphism for the \'etale topology $(Y,x)\simeq (C,0)$.
\item If $K=\bC$, there exists a local analytic isomorphism  $(X,x)\simeq (T,0)$ inducing a local analytic isomorphism  $(Y,x)\simeq (C,0)$.
\end{itemize}
\end{prop}

\begin{proof}
The idea for the following proof is due to C. Chiu. All eventual mistakes are mine. 

It is enough to prove the first assertion. The proof is an adaptation of the proof of  \cite[Cor. 2.6]{ar1}. We recast assumptions as the existence of formal solutions to a system of polynomial equations. Applying Artin approximation, the formal solutions will be approximated by algebraic power series solutions. The approximation will be good enough to give the desired local isomorphisms for the \'etale topology.

We can assume that $X$ is affine, embedded as a closed subvariety of $\bA^N$ and $x$ is the origin. So $X=\spec K[y]/(f)$ and $Y=\spec K[y]/(f,g)$ where $y$ is short-hand notation for $y_1,\ldots,y_N$, $f$ is a tuple $f_1,\ldots, f_r$ and  $g$ is a tuple $g_1,\ldots,g_s$ with  $f_i,g_j\in K[y]$ such that $f_i(0)=g_j(0)=0$, and $(f,g)$ is the ideal generated by all $f_i, g_j$. Also, we can assume that $T=\spec K[x]$ and $C=\spec K[x]/(h)$ with $x=x_1,\ldots, x_n$, where $n=\dim X$, and $h=h_1,\ldots, h_{t}$ with $h_i\in K[x]$ such that $h_i(0)=0$. 

Since the ideal $(h)$ defines a cone, the size of a minimal set of generators for $(h)$ is the same the size of minimal set of generators for $(h)$ in the local ring $k[x]_{(x)}$, and the latter size is equal to 
$$\dim_K ((h)\otimes_{K[ x]}K[ x]_{(x)}/(x)K[ x]_{(x)})= \length_{K[ x]_{(x)}}((h)\otimes_{K[ x]}K[ x]_{(x)}/(x)K[ x]_{(x)})$$ 
by Nakayama Lemma. Here $(x)$ is the ideal generated by the $x_i$. Thus, assuming that  $h_i$ minimally generate $(h)$, we have that $t$ is equal to the above length, and by flatness of completion, 
 
 $$t=\dim_K ((h)\otimes_{K[x]}K\llbracket x\rrbracket/(x))=\length_{K\llbracket x\rrbracket}((h)\otimes_{K[x]}K\llbracket x\rrbracket/(x)).$$


The local isomorphism $\wh{\cO}_{X,0}\simeq \wh{\cO}_{T,0}$ is given by an assigment $y_i\mapsto Y_i$
where $Y=Y_1,\ldots, Y_N$ in $K\llbracket x\rrbracket$ is 
  a power series solution to the system of polynomial equations $f(y)=0$. Since $X$ is smooth we can assume that $y_1,\ldots , y_n$ map to a regular system of parameters in $K\llbracket x\rrbracket$. The assignment being a local isomorphism is  equivalent to the condition that $\det (\pa Y_i/\pa x_j)_{1\le i,j\le n}\not\equiv 0$ modulo $(x)$. The condition that this induces a local isomorphism $\wh{\cO}_{Y,0}\simeq\wh{\cO}_{C,0}$ is equivalent with the condition the  ideals $(g)$ and $(h)$ are mapped to each other. That $(g(Y))$ is  a subideal of $(h)$ is equivalent to the existence of a matrix $A=(A_{ij})_{1\le i\le s, 1\le j\le t}$ in $K\llbracket x\rrbracket$ such that  $g(Y)=A\cdot h(x)$. Then the equality of ideals $(g(Y))=(h)$ is equivalent by Nakayama Lemma with $A$ modulo $(x)$ having  rank $t$ over $K$. This implies in particular that $s\ge t$, hence $A$ must have full rank modulo $(x)$.
Note that $(A,Y)$ is then a solution in $K\llbracket x\rrbracket$ of the system of polynomial equations $h(x)=a\cdot g(y)$ in $K[x][y,a]$, where $a=(a_{ij})_{1\le i\le s, 1\le j\le t}$ are new independent variables.
  
  To summarize, the assumptions are equivalent to having a solution $(A,Y)$ in $K\llbracket x\rrbracket$ to the system of equations in  $K[x][y,a]$ given by
\be\label{eqsysA}
f(y)=0,\quad h(x)=a\cdot g(y),
\ee 
such that
\be\label{eqsyscA}
\det(\pa Y_i/\pa x_j)_{1\le i,j\le n}\not\equiv 0\text{ mod } (x),\quad \rank (A  \text{ mod } (x))=t.
\ee 
 
 By Artin approximation \cite[Cor. 2.1]{ar1}, we can find an affine \'etale neighborhood $(T',0')$ of $0$ in $T$, and
  a solution $(A',Y')$ to (\ref{eqsysA}) in $\cO_{T'}$ such that $(A',Y')\equiv (A,Y)$ modulo $(x)^2$. Then the conditions (\ref{eqsyscA}) are  satisfied for $(A',Y')$ modulo $(x)$. Denote by $C'$ the restriction of $T'$ to $C$. Then $C'$ is given by $(h)\cdot\cO_{T'}$ and $(C',0')$ is an \'etale neighborhood of $0$ in $C$. 
 Let  $K[y]/(f)\to\cO_{T'}$ be the $K$-algebra map given by $Y'$, corresponding to a morphism $T'\to X$. Since  it induces an isomorphisms on the completed local rings $\wh{\cO}_{X,0}\simeq\wh{\cO}_{T',0'}$, $T'\to X$ is an \'etale morphism, cf. \cite[top of p.29]{ar1}. We have $h=A'\cdot g(Y')$. We also have that $A'$ modulo $\mm'$ has full rank $t$, where $\mm'$ is the maximal ideal at $0'$. By flatness of $T'\to T$, 
 $$
 t=\length_{(\cO_{T'})_{\mm'}}((h)\cO_{T'}\otimes_{\cO_{T'}}(\cO_{T'})_{\mm'}/\mm'(\cO_{T'})_{\mm'}) = \dim_K ((h)\cO_{T'}\otimes_{\cO_{T'}}(\cO_{T'})_{\mm'}/\mm'(\cO_{T'})_{\mm'}).
 $$
It follows by Nakayama Lemma that there is an equality of ideals $(g(Y'))=(h)\cO_{T'}$.  Hence the restriction of $T'\to X$ to $Y$, which is \'etale over $Y$, is exactly $C'$. 
  \end{proof}

The proposition can be extended from a pair to a tower of closed embeddings:

\begin{prop}\label{propTower}
Let $X$ be a smooth  $K$-variety, $X\supset Y_1\supset Y_2\supset \ldots Y_m$  closed subschemes, and $x\in Y_m$ a point. Let $T=T_xX$, $C_i=TC_xY_i$, and $0\in C_m$ be the vertex. Suppose there exists a $K$-isomorphism of formal neighborhoods $\widehat{X}_x\simeq \wh{T}_0$ inducing  isomorphisms

$$
\begin{tikzcd}
\wh{X}_x \arrow[d,"\simeq"] \arrow[r, hookleftarrow]& \wh{Y}_{1,x} \arrow[d,"\simeq"] \arrow[r, hookleftarrow]& \wh{Y}_{2,x} \arrow[d,"\simeq"] \arrow[r, hookleftarrow] &  \ldots \arrow[r, hookleftarrow]& \wh{Y}_{m,x} \arrow[d,"\simeq"]\\
\wh{T}_0 \arrow[r, hookleftarrow]& \wh{C}_{1,0} \arrow[r, hookleftarrow]& \wh{C}_{2,0} \arrow[r, hookleftarrow]&  \ldots \arrow[r, hookleftarrow]& \wh{C}_{m,0}.\\
\end{tikzcd}
$$

Then:
\begin{itemize}
\item There exists a local isomorphism for the \'etale topology $(X,x)\simeq (T,0)$ inducing  local isomorphisms for the \'etale topology 

\be\label{eqdiT}
\begin{tikzcd}
(X,x) \arrow[d,"\simeq"] \arrow[r, hookleftarrow]& (Y_1,x) \arrow[d,"\simeq"] \arrow[r, hookleftarrow]& (Y_2,x) \arrow[d,"\simeq"] \arrow[r, hookleftarrow] &  \ldots \arrow[r, hookleftarrow]& (Y_m,x) \arrow[d,"\simeq"]\\
(T,0) \arrow[r, hookleftarrow]& (C_1,0) \arrow[r, hookleftarrow]& (C_2,0) \arrow[r, hookleftarrow]&  \ldots \arrow[r, hookleftarrow]&(C_m,0).\\
\end{tikzcd}
\ee

\item If $K=\bC$, there exists a local analytic isomorphism  $(X,x)\simeq (T,0)$ inducing  local analytic isomorphisms   
in the  diagram (\ref{eqdiT}).
\end{itemize}
\end{prop}

\begin{proof} We give the proof for $m=2$. For higher $m$ the proof is similar. For $m=2$ we start from the proof of Proposition \ref{propArSteroids} by taking $Y=Y_1$, $C=C_1$,  using the objects and the notation introduced there. We can assume $Y_2$ is given by the ideal $(g,  \tilde g)$, where $\tilde g=g_{s+1},\ldots,g_{s_2}\in K[y]$ with $g_j(0)=0$. We can assume $C_2$ is given by the ideal $(h,\tilde h)$ with $\tilde h=h_{t+1},\ldots,h_{t_2}\in K[x]$, where $h_j(0)=0$. As before, we can assume $t_2$ is the minimal number of generators locally at the origin of the ideal $(h,\tilde h)$, that is, $t_2=\dim_K ((h,\tilde h)\otimes_{K[x]} K\llbracket x\rrbracket)$. Let $g^{(2)}$ be the vector of polynomials $g,\tilde g$. Let $h^{(2)}$ be the vector of polynomials $h,\tilde h$. 

The extra assumption in this case comes from the equality of ideals $(g^{(2)}(Y))=(h^{(2)})$.
This is equivalent to the following. There exists a matrix $ A^{(2)}=(\tilde A_{ij})_{1\le i\le s_2,1\le j\le t_2}$ of elements in $K\llbracket x\rrbracket$ such that 
$$
A^{(2)} = \left(
\begin{matrix}
A & O\\
B & \tilde A
\end{matrix}
\right)
$$
where $A$ is the $s\times t$ matrix as before, $O$ is the zero matrix,  and $B$ and $\tilde A$ are matrices of appropriate sizes. The matrix $A^{(2)}$ must in addition satisfy that  its reduction modulo $(x)$ has full rank, $t_2$. This is equivalent to reduction of  $\tilde A$ modulo $(x)$ having full rank, $t_2-t$, by the assumption on $A$. 

We introduce as before new independent variables $b$, $\tilde a$. Let $a^{(2)}$ be the matrix $\big(\begin{smallmatrix}
a & 0\\
 b & \tilde a
\end{smallmatrix}\big)$.
Then, following the proof from above, the assumptions are equivalent to having a solution $(A,B,\tilde A, Y)$ in $K\llbracket x\rrbracket$ to the system of equations in  $K[x][y,a,b,\tilde a]$ given by
\be\label{eqsysA2}
f(y)=0,\quad h(x)=a\cdot g(y),\quad \tilde h(x)=a^{(2)}\cdot\tilde g(y)
\ee 
such that
\be\label{eqsyscA2}
\det(\pa Y_i/\pa x_j)_{1\le i,j\le n}\not\equiv 0\text{ mod } (x),\quad \rank (A  \text{ mod } (x))=t, \quad \rank (\tilde A  \text{ mod } (x))=t_2-t.
\ee 
We proceed as before via Artin approximation. We can find an affine \'etale neighborhood $(T',0')$ of $0$ in $T$, and
  a solution $(A',B', \tilde A',Y')$ to (\ref{eqsysA2}) in $\cO_{T'}$ such that $(A',B', \tilde A',Y')\equiv (A,B,\tilde A,Y)$ modulo $(x)^2$. Then the conditions (\ref{eqsyscA2}) are  satisfied for $(A',B', \tilde A',Y')$. This guarantees as before that the restrictions $C', C_2'$ of $T'\to T$ to $C$, $C_2$, respectively, are \'etale neighborhoods of $0$ in $Y$, $Y_2$, respectively.
\end{proof}

\noindent{\bf Proof of Theorem \ref{thmGenToConeEt}.} It follows from Proposition \ref{propTower} and Theorem \ref{thmGenToCone}. $\hfill\Box$

\section{Generic matrices}\label{secDet}

We review  some results on singularities of spaces of generic  matrices.  We take $K=\bC$. Fix  $a, b\in\bZ_{>0}$. We regard the affine space $\bA^{ab}$ as the space of $b\times a$ matrices.  By $\bz$ we denote the zero matrix in $\bA^{ab}$.  Without loss of generalization, we assume that $0<a\le b$.

\begin{defn}\label{defnGen}$\;$ 
\begin{enumerate}[(1)]

\item The {\it generic matrix} is the matrix $X=(x_{ij})$  of algebraically independent variables $x_{ij}$ with $1\le i\le b$, $1\le j\le a$.

\item For $k\in\bN$ let $J_k= J_k(a,b)$ be the ideal generated by the minors of size $a-k+1$ of the matrix $X=(x_{ij})$. We set $J_0=0$, and  $J_k=(1)$ if  $k\ge a+1$, and this is compatible with convention on minors from \ref{subNot}. The ideals $J_k$ are called {\it generic determinantal ideals}.
 
\item
Let
$$M_k=M_k(a,b):=\{ A\in \bA^{ab}\mid \rank(A)\le a-k\}.$$ 
The spaces $M_k$ are called {\it generic determinantal varieties}. 
\end{enumerate}

\end{defn}

It is well-known that $M_k$ is indeed an affine subvariety of $\bA^{ab}$ and that $J_k$ is the associated radical ideal \cite[II.3]{A+}. Here are some results about the singularities of generic determinantal varieties. Additional results are surveyed in \cite[I]{BD}, where one can also find a review of the terminology from singularity theory used below.

\begin{theorem}\label{thrmGen} Let $1\le k\le a\le b$ be natural numbers and $M_k$ the space of $b\times a$ matrices of rank $\le a-k$. Then:

\begin{enumerate}[(i)]

\item {\rm{(}}\cite[II.2]{A+}{\rm{)}} The variety $M_k$ is isomorphic its tangent cone at $\bz$, it has dimension $(a-k)(b+k)$, and its singular locus is $M_{k+1}$.


\item {\rm{(}}\cite[Prop. 2]{Ke}{\rm{)}} $M_k$ has rational singularities. 

\item {\rm{(}}\cite{Jo}, \cite{Roi}{\rm{)}} The log canonical threshold at $\bz$ of the pair $(\bA^{ab},M_k)$ equals  the global log canonical threshold and equals
$$
\min\left\{\frac{(a-i)(b-i)}{a -k+1-i}\mid i=0,\ldots ,a-k\right\}.
$$



\item {\rm{(}}\cite{Lor}{\rm{)}} The Bernstein-Sato polynomial of the generic determinantal ideal $J_1$ is
$$
\prod_{i=b-a+1}^b (s+i).
$$
The same holds for the local Bernstein-Sato polynomial at $\bz$.

\item {\rm{(}}\cite{Roi}{\rm{)}} If $a=b$, the (global) topological zeta function of the pair $(\bA^{ab},M_k)$ equals the local topological zeta function at the origin and is
$$
\prod_{\al\in \Omega}\frac{1}{1-\al^{-1}s}
$$
where $\Omega$ is the set of poles:
$$
\Omega=\left\{ -\frac{a^2}{a-k+1}, -\frac{(a-1)^2}{a-k}, -\frac{(a-2)^2}{a-k-1},\ldots, - k^2 \right\}.
$$


\item {\rm{(}\cite[4.3]{Jo}, \cite[\S 3]{Sta}\rm{)}} Consider $f_{a-k}:Y_{a-k}\to \bA^{ab}$ the composition of blowups of (strict transforms of) $M_a$, $M_{a-1}$, $\ldots$, $M_{k}$,  in this order.  At each stage this is the blowup of a smooth center in a smooth variety, such that $f_{a-k}$ is a log resolution $(\bA^{ab},M_{k})$. Moreover, the pullback of the ideal $I_{a-k+1}$ defining $M_{k}$ is $\cO_{Y_{a-k}}(-\sum_{i=0}^{a-k}(a-k+1-i)E_i)$, where $E_i$ is the (strict transform of the) divisor introduced by blowing up the (strict transform of) $M_{a-i}$.

\item {\rm{(}\cite{Gaf}\rm{)}} The stratification of $M_k$ given by $M_{t}\setminus M_{t+1}$ with $k\le t$ is a Whitney stratification, and the local Euler obstruction at $\bz$ of $M_k\subset \bA^{ab}$ is $\binom{a}{a-k}$.



\item\label{itMal} {\rm{(}\cite{Mal}\rm{)}} If $a=b$ the minimal discrepancies of $M_k$ along $M_{k+1}$ and, respectively, along a point $w\in M_{k'}\setminus M_{k'+1}$ with $k\le k'\le a$ are:
$$
\mld(M_{k+1};M_k) = k+1,\quad \mld(w;M_k) = a^2-kk'.
$$

\end{enumerate}

\end{theorem}

\begin{rmk}
For {\color{hot}(v)} only the formula for the global topological zeta function is  given in \cite{Roi}. However, the description in \cite{Roi} in terms of pre-partitions of the strata of jet schemes allows the computation of the local topological, in fact even motivic, zeta function at the origin as well. 
\end{rmk}

\section{Proof of  Theorem \ref{thrmLCThrk}}\label{Proof}

We prove now  Theorem \ref{thrmLCThrk}.
By Theorem \ref{thmGenToConeEt}, there exists a tower of cartesian diagrams

\be\label{eqBlubli}
\begin{tikzcd}
(\cM,E) \arrow[d,hookleftarrow]& (X,x) \arrow[l] \arrow[r] \arrow[d,hookleftarrow]& (M ,\bz) = (T_E\cM,\bz)\arrow[d,hookleftarrow]\\
(\cV_1,E) \arrow[d,hookleftarrow]& (X_1,x) \arrow[l] \arrow[r] \arrow[d,hookleftarrow]& (M_1,\bz) = (TC_E\cV_1,\bz) \arrow[d,hookleftarrow]\\
(\cV_2,E) \arrow[d,hookleftarrow]& (X_2,x) \arrow[l] \arrow[r] \arrow[d,hookleftarrow]& (M_2,\bz) =(TC_E\cV_2,\bz)\arrow[d,hookleftarrow]\\
\vdots \arrow[d,hookleftarrow]&\vdots \arrow[d,hookleftarrow]& \vdots \arrow[d,hookleftarrow]\\
(\cV_l,E) & (X_l,x) \arrow[l] \arrow[r] & (M_l,\bz)=(TC_E\cV_l,\bz) \\
\end{tikzcd}
\ee
where: the horizontal maps  are \'etale, sending $E\mapsfrom x\mapsto \bz$, inducing isomorphisms of residue fields at $E,x,\bz$; the vertical maps are closed embeddings of subschemes; $$\cM=\cM_{n,d},
\cV_k=\cV_k(F), M=\HH^1(C,E\otimes E^\vee), \bz \text{ is the origin};$$ $M_k$ is the closed subscheme defined by the ideal generated by the minors of size $l-k+1$ of matrix of linear forms on $\HH^1(C,E\otimes E^\vee) $ determined by $\pi_{E,F}$. 
By the injectivity of $\pi_{E,L}$ this matrix is generic  of size $l'\times l$, cf. Definition \ref{defnGen}.

To simplify notation, we denote by the same symbols, and work with them so from now on, the restriction of the diagram to two Zariski open neighborhoods of $E$ and $x$, respectively in $\cM_{n,d}$ and $X$, respectively. By shrinking these open neighborhoods, we can and will assume that $\cM, \cV_k, X, X_k$ are connected.

(\ref{qo}) These properties follow from Theorem \ref{thrmGen} {\color{hot}(i)-(ii)}, with $a=l$, $b=l'$. 
Note that $M$ and $M_k$ here are the same as $\bA^{ab}$ and $M_k$, respectively, from Theorem \ref{thrmGen} up to the product with an affine space of dimension equal to $h^1(E\otimes E^\vee) - ll'$. The codimension of $M_k$ here agrees with the codimension of $M_k$ from Theorem \ref{thrmGen}.
Since $\cM, \cV_k, X, X_k$ are connected and $M, M_k$ are reduced, irreducible and have rational singularities, it follows that $\cM, \cV_k, X, X_k$ are also reduced, irreducible and have rational singularities.

(\ref{qiii}) The local Bernstein-Sato polynomial, see \cite[Def. 5.11]{BD}, is an embedded (if one does not shift by codimension) analytic invariant. Hence it is the same for the three triples $(\cM,\cV_k,E)$, $(X,X_k,x)$, $(M,M_k,\bz)$. The claim then follows from Theorem \ref{thrmGen} {\color{hot}(iv)}. 

(\ref{qix}) The local lct is an local embedded analytic invariant, since it can be read from the local Bernstein-Sato polynomial, see \cite[Thm. 5.12]{BD}. Hence the local lct is the same for the  triples $(\cM,\cV_k,E)$, $(X,X_k,x)$, $(M,M_k,\bz)$. The claim then follows from Theorem \ref{thrmGen} {\color{hot}(iii)}.

(\ref{qvi}) One performs the same sequence of blowups for the three vertical towers corresponding to $\cM,X,M$ in the diagram (\ref{eqBlubli}) to obtain a diagram of cartesian squares 
$$
\begin{tikzcd}
Y \arrow[d, "f"] & Y_X \arrow[d, "f_X"] \arrow[r] \arrow[l]& Y_M \arrow[d, "f_M"] \\
\cM & X \arrow[r] \arrow[l]& M
\end{tikzcd}
$$
 with the horizontal maps \'etale.
Each step performed to obtain $f_X$  fits into a cartesian diagram with the same step for $f$ and $f_M$. Since the base change of an \'etale map is \'etale, the center to be blown up at each step is smooth since it is so at that step for $f_M$ by Theorem \ref{thrmGen} {\color{hot}(vi)}, and it is irreducible since it is the strict transform of an irreducible variety from below.  Since we blow up a smooth center in a smooth variety, we introduce each time only one new exceptional divisor. Hence there is a one-to-one correspondence between the set of prime divisors $E_i$ obtained for  $f$  and those obtained for $f_X$ and respectively  $f_M$, which we denote by $E_{X,i}$ and $E_{M,i}$, respectively.
By \'etale base change, there is an equality of orders of vanishing 
$$
\ord_{E_i}(\cV_k)=\ord_{E_{X,i}}(X_k) = \ord_{E_{M,i}}(M_k),
$$
and the latter are given by  Theorem \ref{thrmGen} {\color{hot}(vi)}. Finally, since the morphisms of schemes
$$
f^{-1}(\cV_k )\leftarrow f_X^{-1}(X_k) \to f_M^{-1}(M_k)
$$
are obtained by base change, they are \'etale. All three are schemes associated to effective divisors. Since $f_M^{-1}(M_k)$ has simple normal crossings support by  Theorem \ref{thrmGen} {\color{hot}(vi)}, and snc is a notion defined at each point using the completion at that point, it follows that the other two have also simple normal crossings support. Hence $f$ is a log resolution of $(\cM,\cV_k)$.

(\ref{qvii}) These are local properties of the underlying analytic topology, which thus do not change under \'etale maps. The claim then follows from Theorem \ref{thrmGen} {\color{hot}(vii)}.


(\ref{qiv}) The motivic, and thus also the topological, zeta function of the pair $(\cM,\cV_k)$ at $E$ depends only on the completions $(\wh{\cM}_E,\wh{(\cV_k)}_E)$. Hence the local topological zeta function equals that of $(M,M_k)$ at $\bz$. The latter is given in Theorem \ref{thrmGen} {\color{hot}(v)}. 

(\ref{qv})  By the case $k=1$ of  (\ref{qiv}) we conclude that all poles of the local topological zeta function of $(\cM,\cV_1)$ at $E$ are simple poles. By (\ref{qiii}) they are all roots of the local Bernstein-Sato polynomial of  $(\cM,\cV_1)$ at $E$.

(\ref{xi}) The local \'etale embedded isomorphism implies that for a suitable Zariski open neighborhood $U$ of $E$ in $\cM$,  $\mld(\cV_{k+1}\cap U,\cV_k\cap U)= \mld(M_{k+1},M_k)$, where for the latter pair we do not need to shrink to an open neighborhood of $\bz$ since $M_k$ are cones. Then the claim follows from Theorem \ref{thrmGen} (\ref{itMal}). Similarly, $\mld(E',\cV_k\cap U)= \mld(w,M_k)$ where $E'\in U\cap(\cV_{k'}\setminus \cV_{k'+1})$ and $w\in M_{k'}\setminus M_{k'+1}$ and the claim follows from Theorem \ref{thrmGen} (\ref{itMal}) again.
$\hfill\Box$

\end{document}